\documentclass[12pt]{amsart}
\usepackage{amsmath,amsthm,latexsym,amscd,amsbsy,amssymb,url}
\usepackage[all]{xypic}
 \relax

\chardef\bslash=`\\ 

\makeatletter
\def\verbatim{\interlinepenalty\@M \@verbatim
  \leftskip\@totalleftmargin\advance\leftskip2pc
  \frenchspacing\@vobeyspaces \@xverbatim}
\makeatother
\newtheorem{thm}{Theorem}[section]
\newtheorem{cor}[thm]{Corollary}
\newtheorem{lem}[thm]{Lemma}
\newtheorem{pro}[thm]{Proposition}
\newtheorem{defin}[thm]{Definition}

\newtheorem{ex}[thm]{Example}
\newtheorem{que}[thm]{Question}

\begin{document}


\title
{A Note On Minimal Separating Function Sets}
\author{Raushan Buzyakova}
\email{Raushan\_Buzyakova@yahoo.com}
\address{Miami, Florida, U.S.A.}
\author{Oleg Okunev}
\email{oleg@servidor.unam.mx}
\address{Facultad de Ciencias Físico-Matematicas, Benemérita Universidad Autonoma de Puebla, Apdo postal 1152, Puebla, Puebla CP 72000, Mexico}


\begin{abstract}{
We study point-separating function sets that are minimal  with respect to the property of being separating.
We first show that for a compact space $X$  having a minimal separating function set in $C_p(X)$ is equivalent to having a minimal separating collection of functionally open sets in $X$. We also identify a nice visual property of  $X^2$ that may be responsible for the existence of a minimal  separating function family for $X$ in $C_p(X)$. We then discuss various questions and directions around the topic.
}
\end{abstract}

\maketitle
\markboth{R. Buzyakova and O. Okunev}{A Note On Minimal Separating Function Sets}
{ }

\par\bigskip
\section{Introduction}\label{S:introduction}
\par\bigskip\noindent
In this discussion we assume that all spaces are Tychonov. In notations and terminology we will follow \cite{Eng} and \cite{Arh}
Recall that a function set $F\subset C_p(X)$ is {\it  point separating} if for every distinct $x,y\in X$ there exists $f\in F$ such that $f(x)\not = f(y)$. We will refer to such sets as {\it separating function sets}.

\par\smallskip\noindent
Ample research has been done on separating function sets. It is natural to look for separating sets that have additional nice properties such as compactness, metrizability, etc. Note that it  is quite likely that a randomly rendered separating function  set contains a proper separating subset with perhaps better topological properties. With this in mind we would like to take a closer look  at {\it }minimal separating function sets. 

\par\bigskip\noindent
{\bf Definition} {\it
A set $F\subset C_p(X)$ is a minimal separating function set if $F$ is separating for  $X$ and no proper subset of $F$ is separating.  If $F$ is separating, closed (compact),  and has not no proper closed (compact) separating  subset, $F$ is a minimal    closed (compact) separating function set.
}
\par\bigskip\noindent
In this paper we are concerned with the following general problem:

\par\bigskip\noindent
{\bf Problem.} {\it
What properties of $X,X^n$ are sufficient/necessary for $X$ to have a  minimal separating function set in $C_p(X)$?  A  minimal closed (compact) separating function  set in $C_p(X)$?
}
\par\bigskip\noindent
In the first part of our study we will concentrate on minimal separating functions sets and then we will zoom into minimal closed (compact) separating function sets.

\par\bigskip
\section{Equivalence}\label{S:equivavlence}
\par\bigskip\noindent
In this section we would like to make an observation that for a compact space $X$ having a minimal separating function set  in $C_p(X)$ is equivalent to having a minimal point-separating family of functionally open subsets in $X$. One part of this equivalence  is quite obvious but the other requires some work.  Recall that   a family  of subsets of $X$ is called {\it point-separating} if for any two distinct elements $x$ and $y$ in $X$ there exists an element in the family that contains $x$ or $y$ but not both. Such a family is minimal if no proper subfamily of it is point-separating. 
\par\bigskip\noindent
\begin{thm}\label{thm:topologyToCp}
If $X$ has a minimal separating family of functionally open sets then $X$ has a minimal separating function set.
\end{thm}
\begin{proof}
Let $\mathcal O$ be a family as in the hypothesis. For each $O\in \mathcal O$, fix $f_O: X\to [0,1]$ such that $f^{-1}((0,1])=O$. Clearly, $\{f_O:O\in \mathcal O\}$ is a minimal separating function set.
\end{proof}

\par\bigskip\noindent
To reverse the statement of Theorem \ref{thm:topologyToCp}, we  first prove four very technical assertions.

\par\bigskip\noindent
\begin{pro}\label{pro:discreteinsquare}
Let $X$ have a $\tau$-sized minimal separating function set. Then there exists a $\tau$-sized $D\subset (X\times X)\setminus \Delta_X$ that is discrete and closed in $(X\times X)\setminus \Delta_X$.
\end{pro}
\begin{proof}
Fix a minimal separating function  set $F\subset C_p(X)$ of cardinality $\tau$. By minimality, for each $f\in F$ we can fix distinct $x_f,y_f$ in $F$ that are separated by $f$ but not by any other function in $F\setminus \{f\}$.
The set $D=\{\langle x_f,y_f\rangle : f\in F\}$ is a subset of $(X\times X)\setminus \Delta_X$. To show that $D$ is a closed discrete subset in the off-diagonal, fix $\langle x,y\rangle$ in $X\times X$ with $x\not = y$. We need to show that $\langle x,y\rangle$ can be separated from $D\setminus \{\langle x,y\rangle\}$ by an open neighborhood.

Since $F$ is separating, there exists $h\in F$  such that $h(x)\not = h(y)$.
Let $\epsilon = |h(x)-h(y)|/3$. Recall that $h$  does not separate $x_f$ from $y_f$ if $h\not = f$. Therefore, 
$(h(x)-\epsilon, h(x)+\epsilon)\times  (h(y)-\epsilon, h(y)+\epsilon)$ is  a desired open neighborhood of $\langle x,y\rangle$.
\end{proof}

\par\bigskip\noindent
In our  discussion, we will reference  not only the statement of Proposition \ref{pro:discreteinsquare} but its argument too, namely, the definition of $D$ in terms of elements of $F$.
\par\bigskip\noindent
\begin{lem}\label{lem:sequencetozero}
There exists a collection $\{U_n:n=1,2,...\}$ of open subsets of $\mathbb R$ with the following properties:
\begin{enumerate}
	\item $\{U_n:n=1,2,...\}$ is point separating,
	\item $U_n$ contains $0$ for each $n=1,2,...$,
	\item $U_n$ is the only element of the collection that separates $0$ from $1/n$.
\end{enumerate}
\end{lem}
\begin{proof}
Put $S= \{0\}\cup\{1/n:n=1,2,...\}$.  Since $\mathbb R$ is second-countable and $S$ is closed in $\mathbb R$,  we can fix $\mathcal P = \{\langle A_n, B_n\rangle:n=1,2,..\}$ with the following properties:
\begin{description}
	\item[\rm P1] $A_n,B_n$ are open in $\mathbb R$ and their closures miss $S$,
	\item[\rm P2] $\bar A_n\cap \bar B_n =\emptyset$,
	\item[\rm P3] For any distinct $x,y\in \mathbb R\setminus S$ there exists $n$ such that $x\in A_n$ and $y\in B_n$.
\end{description}
By P1, for each $n$, we can fix an open neighborhood $O_n$ of $S$ that misses $\bar A_n\cup \bar B_n$. Put $U_n = [O_n\setminus \{1/n\}]\cup A_n$. Let us show that $\mathcal U=\{U_n:n=1,2,...\}$ is as desired.
\par\smallskip\noindent
Properties (2) and (3) are incorporated in the definition of $U_n$'s. To show that $\mathcal U$ is point separating, fix distinct $x,y\in \mathbb R$. We have three cases:
\begin{description}
	\item[\rm Case ($x,y\not \in S$)] By P3, there exists $n$ such that $x\in A_n$ and $y\in B_n$. Then $x\in U_n$. Since $\bar O_n$ and $\bar A_n$ miss $\bar B_n$, we conclude that $U_n$ misses $B_n$. Therefore, $y\not \in U_n$.
	\item[\rm Case ($x\not \in S, y\in S$)] Pick any positive integer $K$ such that $y\not = 1/K$ and $x\not = K$. Since 
$x,K\not\in S$ and $K\not = x$, by P3, there exists $n$ such that $K\in A_n$ and $x\in B_n$. Then $U_n$ contains $y$ but does not contain $x$.
	\item[\rm Case ($x,y\in S$)] We  may assume that $x = 1/n$. Then $U_n$ contains $y$ but not $x$.
\end{description}
The proof is complete.
\end{proof}

\par\bigskip\noindent
\begin{lem}\label{lem:twosequencestozero}
There exists a collection $\{U_n:n=1,2,...\}$ of open subsets of $\mathbb R$ with the following properties:
\begin{enumerate}
	\item $\{U_n:n=1,2,..\}$ is point separating,
	\item $U_n$ is the only element of the collection that separates $-1/n$ from $1/n$.
\end{enumerate}
\end{lem}
\begin{proof}
Put $S= \{0\}\cup\{\pm 1/n:n=1,2,...\}$.  Since $\mathbb R$ is second-countable and $S$ is closed in $\mathbb R$,  we can fix $\mathcal P = \{\langle A_n, B_n\rangle:n=1,2,..\}$ with the following properties:
\begin{description}
	\item[\rm P1] $A_n,B_n$ are open in $\mathbb R$ and their closures miss $S$,
	\item[\rm P2] $\bar A_n\cap \bar B_n =\emptyset$,
	\item[\rm P3] For any distinct $x,y\in \mathbb R\setminus S$ there exists $n$ such that $x\in A_n$ and $y\in B_n$.
\end{description}
By P1, for each $n$, we can fix an open neighborhood $O_n$ of $S$ that misses $\bar A_n\cup \bar B_n$.
 Put $U_1 = [O_1 \setminus \{1\}]\cup A_1$. If $n>1$, put
$U_n = [O_n \setminus \{1/n,\pm 1/(n-1)\}]\cup A_n$. Let us show that $\mathcal U=\{U_n:n=1,2,...\}$ is as desired.
\par\smallskip\noindent
Property (2) in the conclusion of the lemma's statement  is  incorporated in the definition of $U_n$'s. To show that $\mathcal U$ is point separating, fix distinct $x,y\in \mathbb R$. If $x$ or $y$ is in $\mathbb R\setminus S$, then the argument is as in Lemma \ref{lem:sequencetozero}. We now assume that $x,y\in S$.  We have three cases:
\begin{description}
	\item[\rm Case ($x=0$)] Then $|y| = 1/(n-1)$ for some $n$. Hence, $U_n$ separates $x$ and $y$.
	\item[\rm Case ($x=1/n, y=-1/n$)] Then $U_n$ separates $x$ and $y$.
	\item[\rm Case ($|x|=1/n,|y|=1/m,n< m$)] Then $U_{m+1}$ does not contain $\pm 1/m=\pm y$ but contains $\pm 1/n=\pm x$.
\end{description}
The proof is complete.
\end{proof}

\par\bigskip\noindent
Note that $\{U_n\}_n$'s constructed in Lemmas \ref{lem:sequencetozero} and \ref{lem:twosequencestozero} are  minimal point-separating families of functionally open sets of $\mathbb R$. The participating sequences can be replaced by any non-trivial convergent sequences as well as $\mathbb R$ can be replaced by any non-discrete second-countable space.

\par\bigskip\noindent
Either of the above two lemmas imply the following:
\par\bigskip\noindent
\begin{cor}\label{pro:separablemetric}
 Any separable metric space $X$ has a minimal separating family of functionally open  sets.
\end{cor}

\par\bigskip\noindent
We are now ready to prove the reverse of  of Theorem \ref{thm:topologyToCp}.
\par\bigskip\noindent

\begin{thm}\label{thm:CpToTopology}
Let $X$ have  a minimal separating function set. Then $X$ has a minimal separating collection of functionally open set.
\end{thm}
\begin{proof}
Fix a minimal separating function set $F\subset C_p(X)$. By Corollary \ref{pro:separablemetric}, we may assume that $F$ is  infinite. Partition $F$ into $\{F_\alpha:\alpha<|F|\}$ so that each $F_\alpha$ is countably infinite.

\par\smallskip\noindent
For each $\alpha<|F|$ we will define $\mathcal U_\alpha$ so that $\mathcal U=\cup \{\mathcal U_\alpha:\alpha<|F|\}$ will be a minimal point-separating family of functionally open subsets of $X$.  First, for each $f\in F$, fix $x_f$ and $y_f$ in  $X$ that are separated by $f$ and no any other member of $F$.

\par\medskip\noindent
\underline {\it Construction of $\mathcal U_\alpha$:} Put $h_\alpha = \Delta \{f: f\in F_\alpha\}$. Since $h_\alpha(X)$ is compact, there exists $\{f_n:n\in \omega\}\subset F_\alpha$ and $L\in h_\alpha(X)$ such that $\lim\limits_{n\to \infty} x_{f_n} = L$. By the argument of Proposition \ref{pro:discreteinsquare}, the set $\{\langle x_f,y_f\rangle:f\in F\}$ is closed and discrete in $(X\times X)\setminus \Delta_X$. Therefore,  $\lim\limits_{n\to \infty} y_{f_n} = L$ too. Since $x_f\not = y_f$ for each $f$ we may assume that $\{x_{f_n}:n\in \omega\}\cap \{y_{f_n}:n\in \omega\}=\emptyset$ and that no $y_{f_n}$ is equal to $L$. We then have two cases:
\par\medskip\noindent
{\it Case 1.} $x_{f_n}=L$ for all $n$;
\par\noindent
{\it Case 2.} $x_{f_n}\not = x_{f_m}\not = L$ for all $n\not = m$.
\par\medskip\noindent
 Independently on which case takes place, by Lemmas \ref{lem:sequencetozero} and \ref{lem:twosequencestozero}, we can find a collection $\mathcal O_\alpha=\{O_\alpha^n:n\in \omega\}$ of open sets in $h_\alpha(X)$ that separates points of $h_\alpha(X)$ so that only $O^n_\alpha$ separates $x_{f_n}$ from $y_{f_n}$. Put $\mathcal U_\alpha=\{h^{-1}(O):O\in \mathcal O_\alpha\}$.

\par\medskip
Let us show that $\mathcal U=\cup \{\mathcal U_\alpha:\alpha<|F|\}$ is point separating and minimal with respect to this property.

To show that $\mathcal U$ is point separating, fix distinct $x,y\in X$. Since $F$ is a point separating function set, there exists $f\in F_\alpha\subset F$ that separates $x$ and $y$. Then $h_\alpha(x)\not = h_\alpha(y)$. Since $\mathcal O_\alpha$ is a point separating family of open subsets of $h_\alpha(X)$, there exists $O\in \mathcal O_\alpha$ that contains exactly one of  $h_\alpha(x)$ and $h_\alpha(y)$. Hence, $h^{-1}_\alpha(O)\in \mathcal U_\alpha\subset \mathcal U$ separates $x$ and $y$.

To show  minimality of $\mathcal U$ with respect to being point separating, fix $U\in \mathcal U$. Then $U=h^{-1}(O)$ for some $O\in \mathcal O_\alpha$. Then there exists $f\in F_\alpha$ such that $O$ is the only element that separates $x_f$ and $y_f$. 
Since $f$ is the only function in $F$ that separates $x_f$ and $y_f$, we conclude that $U$ is the only element of $\mathcal U$ that separates $x_f$ and $y_f$.
\end{proof}

\par\bigskip\noindent
Theorems \ref{thm:topologyToCp} and \ref{thm:CpToTopology} imply the promised equivalence.
\begin{thm}\label{thm:criteriongeneral}
A compact space $X$ has a minimal separating function set in $C_p(X)$ if and only if $X$ has  a minimal  separating family of  functionally open subsets in $X$.
\end{thm}

\par\bigskip
\section{Minimal Separating Function Sets}\label{S:minimalgeneral}

\par\bigskip\noindent
In this section the property of having a  minimal separating function sets will be often  referred to as {\it the property} and will be studies exclusively from the point of view of $C_p$-theory despite Theorem \ref{thm:criteriongeneral}. The rational behind this choice is that any such set is necessarily discrete in itself (Proposition \ref{pro:minimalthendiscrete}) and discrete separating sets have attracted  attention of many  $C_p$-enthusiasts. 
\par\bigskip\noindent
First observe that a space  can have minimal separating function sets of different cardinalities. 
Indeed, one can easily construct a one- and a two-element minimal separating function  families for  $X=\{1,2,3\}$. While cardinality is not a common property among minimal function sets, the absence of cluster points is.

\par\bigskip\noindent
\begin{pro}\label{pro:minimalthendiscrete}
Any minimal separating function set is discrete in itself.
\end{pro}
\begin{proof}
Let $F\subset C_p(X)$ be a minimal separating set. Fix any $f\in F$. Since $F$ is minimal, there exist distinct $x,y\in X$ that are separated by $f$ but not by any other member of $F$. Put $\epsilon = |f(x)-f(y)|/3$. Then $U_=\{g: |g(x)-f(x)|<\epsilon, |g(y)-f(y)|<\epsilon\}$ is an open neighborhood of $f$ that misses $F\setminus \{f\}$.
\end{proof}

\par\bigskip\noindent
The statement of \ref{pro:minimalthendiscrete} is the main reason we consider our study from the point of view of $C_p$-theory. In \cite{BO} it was proved that if $X$ is a zero-dimensional space of pseudoweight $\tau$ of uncountable cofinality, then $X$ has a discrete separating function set of size $\tau$  if and only if $X^n$ has a discrete set of size $\tau$ for some $n$. This and Proposition \ref{pro:discreteinsquare} prompt the following question.

\par\bigskip\noindent
\begin{que}\label{que:zerodimcriterion}
Let $X$ be a zero-dimensional space of (pseudo)weight $\tau$ and let $(X\times X)\setminus \Delta_X$ has a closed discrete subspace of size $\tau$. Does $C_p(X)$ have a minimal separating set?
\end{que}

\par\bigskip\noindent
This observation naturally prompts a question of whether our property is equivalent to having a discrete separating function set.
We will next identify  a naive example among well-known  spaces that does not have the property under discussion but has a discrete function separating set.

\par\bigskip\noindent
\begin{ex}\label{ex:betaomega}
The space $\beta \omega$ has a discrete  separating function set  but no minimal separating function set.
\end{ex} 
\begin{proof} Recall that $\beta\omega$ is a zero-dimensional compactum that has a $2^\omega$-sized separating family of clopen sets and a discrete in itself subset of cardinality $2^\omega$. By \cite[Theorem 2.8 and Corollary 2.19]{BO}, $\beta\omega$ has a discrete separating function set. Since $(\beta \omega\times \beta\omega)\setminus \Delta_{\beta \omega}$ is countably compact, it cannot contain a discrete closed subset. Hence, by Proposition \ref{pro:discreteinsquare}, $\beta \omega$ does not have a minimal separating function set.
\end{proof}

\par\bigskip\noindent
Now that we have established the existence of spaces without the property, it would be nice to isolated significant classes of spaces that have the property.  We have already proved that separable metric spaces have the property (Theorem \ref{thm:topologyToCp} and Corollary \ref{pro:separablemetric}). In fact, all metric spaces have the property. We will derive it from a more general statement that we will prove next.  For this, by $a(X)$  we denote the largest cardinal (if exists) for which there exists a closed subset $A\subset X$ that contains at most one non-isolated point. Recall that $w(X)$ denotes the weight of $X$.

\par\bigskip\noindent
\begin{thm}\label{thm:sufficient}
Let $X$ be a  normal space and $a(X)=w(X)$. Then $X$ has a minimal separating function set.
\end{thm}
\begin{proof}
Let $S$ be a closed subset of $X$ with at most one non-isolated point and of size equal to $w(X)$. If $S$ has a non-isolated point, denote it  by $p^*$. Otherwise, give this name to an arbitrary point of $S$. Since $S$ is closed, we can fix a $w(X)$-sized family $\mathcal A$ of closed sets with the following properties:
\begin{description}
	\item[\rm A1] For any distinct $x,y\in X\setminus S$ there exist disjoint $A_x,A_y\in\mathcal A$ containing $x$ and $y$, respectively.
	\item[\rm A2] $\bigcup\mathcal A =X\setminus S$.
\end{description}
\par\smallskip\noindent
Let $\mathcal P$ be the set of all unordered pairs of disjoint elements of $\mathcal A$. Enumerate $\mathcal P$ as $\{\{A_\alpha,B_\alpha\}:\alpha<w(X)\}$ and enumerate $S\setminus \{p^*\}$ as $\{p_\alpha:\alpha<w(X)\}$.
Next for each $\alpha<w(X)$, fix a continuous function $f_\alpha:X\to \mathbb R$ that has the following properties:
\begin{description}
	\item[\rm F1] $f_\alpha(S\setminus \{p_\alpha\}) = \{0\}$,
	\item[\rm F2] $f_\alpha(p_\alpha) = 1$,
	\item[\rm F3] $f_\alpha(A_\alpha) = \{1/3\}$,
	\item[\rm F4] $f_\alpha(B_\alpha)=\{2/3\}$.
\end{description}
Such a function exists because $A_\alpha, B_\alpha, S\setminus \{p_\alpha\}$, and $\{p_\alpha\}$ are mutually disjoint closed subsets of a normal space. By A1-A2 and F1-F4, $\{f_\alpha\}_\alpha$ is separating.  By F1-F2, $f_\alpha$ is the only function that separates $p^*$ from $p_\alpha$. Hence, $\{f_\alpha\}_\alpha$ is a minimal separating family of functions.
\end{proof}
\par\bigskip\noindent
\begin{cor}\label{cor:metric}
Every metric space has a minimal separating function set.
\end{cor}

\par\bigskip\noindent
Our next goal is to investigate the behavior of the property within standard structures and under standard operations. First, we will show that any space can be embedded into a space with the property as a closed subset or even as clopen subset. For our next statement, recall that $XX'$ is a standard notation for the Alexandroff double $XX'$ of a  space $X$.

\par\bigskip\noindent
\begin{pro}\label{pro:double}
The Alexandroff double of any  infinite space $X$ has a minimal separating set of functions of cardinality $|X|$.
\end{pro}
\begin{proof}
Put  $\tau=|X|$. Since $X$ is Tychonoff and $|X|=\tau$, there exists a separating family $F\subset C_p(X)$ that has cardinality $\tau$. Arrange elements of $F$ into a sequence $\{f_\alpha: \alpha<\tau\}$ so that every $f$ appears in the sequence infinitely many times. Enumerate elements of $X$ as $\{x_\alpha:\alpha<\tau\}$. Next, for each $\alpha<\tau$, define $g_\alpha:XX'\to\mathbb R$ as follows:
$$
g_\alpha(p) = \left\{
       	 \begin{array}{ll}
           	 f_\alpha(p) & p \in  X\\
            	 f_\alpha(x_\beta)& p=x_\beta' \ for\ some\  \beta\not = \alpha\\
	  	 f_\alpha(x_\alpha)+1 & p=x_\alpha'
        	\end{array}
   		 \right.
$$

Clearly, $g_\alpha$ is continuous and coincides with $f_\alpha$ on $X$. Let us show that $G=\{g_\alpha:\alpha<\tau\}$ is a minimal separating family of functions for $XX'$.

To prove that $G$ separates points of $XX'$, fix arbitrary distinct $p,q\in XX'$. Either both points are in $X$, or both are in $X'$, or one is in $X$ and the other is in $X'$. Let us consider these three cases separately.
\begin{description}
	\item[\rm Case ($p, q\in X$)] The conclusion follows from the facts that  $F$ separates points of $X$ and that $g_\alpha$'s are extensions of $f_\alpha$'s.
	\item[\rm Case ($p, q\in X'$)] There exist distinct $\alpha,\beta$ such  that $p=x_\alpha',q=x_\beta'$. Since $F$ separates points of $X$, there exists $\gamma$ such that $f_\gamma(x_\alpha)\not = f_\gamma(x_\beta)$. Since $f$ appears in the enumeration infinitely many times, we may assume that $\gamma\not \in \{\alpha, \beta\}$. Then $g_\gamma(p)\not = g_\gamma(q)$.
	\item[\rm Case ($p\in X, q\in X'$)] There exist  $\alpha,\beta$ such  that $p=x_\alpha,q=x_\beta'$. If $\alpha=\beta$, then $g_\alpha$ separates $p$ and $q$. Otherwise, we proceed as in Case 2.
\end{description}
Thus, $G$ is point-separating. Finally, only $g_\alpha$ separates $x_\alpha$ from $x_\alpha'$. Therefore, $G$ is minimal.
\end{proof}

\par\bigskip\noindent
An argument analogous to that of Proposition \ref{pro:double} proves  the following statement.
\par\bigskip\noindent
\begin{pro}\label{pro:doublediscrete}
Let $X$ be an infinite space and let $D_X$ be a discrete space of cardinality $|X|$. Then $X\oplus D_X$ has a minimal separating set of functions of cardinality $|X|$.
\end{pro}

\par\bigskip\noindent 
Propositions \ref{pro:double} and \ref{pro:doublediscrete} may create an illusion that any kind of doubling guarantees the presence of the studied property in the resulting "double-space". However, the most basic doubling, namely $X\times \{0,1\}$ need not guarantee any such "improvement". Indeed, $\beta\omega\times \{0,1\}$ is homeomorphic to itself, and therefore, does not have the property by  Example \ref{ex:betaomega}. This observation prompts the following.

\par\bigskip\noindent
\begin{que}
Let $X\times \{0,1\}$ have a minimal separating function set. Does $X$ have such a set?
\end{que}

\par\bigskip\noindent
Statements of Example \ref{ex:betaomega} and Proposition \ref{pro:doublediscrete}  imply that the property is not inherited by closed, open, and even clopen subspaces.  The following question may still be answered in affirmative.

\par\bigskip\noindent
\begin{que}
Let $X$ have a minimal separating family of functions. Does every  open dense subset have the property?
\end{que} 

\par\bigskip\noindent
Our next observation balances some of the negative statements above.

\par\bigskip\noindent
\begin{thm}\label{thm:product}
Let $X_\alpha$ have a minimal separating function set for each $\alpha\in A$. Then, $\prod_{\alpha\in A}X_\alpha$ has a minimal separating function set.
\end{thm}
\begin{proof} Fix a minimal separating family $\mathcal F_\alpha$ for each 
 $X_\alpha$. Put  $\mathcal F=\{f\circ p_\alpha: 
\alpha\in A, f\in\mathcal F_\alpha\}$. Clearly, this family is 
separating (because for any two distinct points $x$, $y$ in the product 
there is an $\alpha\in A$ such that $p_\alpha(x)\neq p_\alpha(y)$, and 
there is an $f\in\mathcal F_\alpha$ with $f(p_\alpha(x))\neq 
f(p_\alpha(y))$. This family is minimal, because
if we remove some $f\circ p_\alpha$, then there are $a,b\in X_\alpha$ 
that are
not separated by $\mathcal F_\alpha\setminus \{f\}$. Now let $x$ and $y$ 
be points of the product whose $\alpha$th coordinates are $a$ and $b$, 
and all the remaining coordinates are equal. Then the only projection 
that separates
these points is $p_\alpha$, and the images under $p_\alpha$ are not 
separated in $X_\alpha$ by $\mathcal F\setminus \{f\}$, and that's it.
\end{proof}

\par\bigskip\noindent
Next, let us look for traces of the property in  images and pre-images. Since any discrete spaces has the property, the property is not preserved by continuous maps, even bijective ones. Statements of Example \ref{ex:betaomega} and Proposition \ref{pro:double}  imply that the property is not inherited by taking the inverse image under a continuous injection. On a positive note,  the domain of a continuous bijection has the property if the range does. Indeed, let $F$ be a minimal separating function set for  $X$ and let  $h:Y\to X$ be  a continuous bijection. Put $G=\{f\circ h: f\in F\}$. Since $h$ is injective, $G$ is a separating family for $Y$. Since $F$ is minimal, for each $f\in F$ we can fix distinct $x_1,x_2\in X$ that are separated by $f$ but not any other function in $F$. By surjectivity, $y_1=h^{-1}(x_1),y_2=h^{-1}(x_2)$ are defined. Clearly, $y_1$ and $y_2$ are separated by $f\circ h$ but not any other function in $G$.  Let us record our observation as a statement.

\par\bigskip\noindent
\begin{pro}\label{pro:bijection}
If $X$ admits a continuous bijection onto a space with a minimal separating set of functions, then $X$ has such a set too.
\end{pro}

\par\bigskip\noindent
Our discussion prompts the following question.

\par\bigskip\noindent
\begin{que}
Let $X$ have a minimal separating set of functions and let $Y$ be $t$-equivalent to $X$. Does $Y$ have such a set?
\end{que}

\par\bigskip
\section{A Glance at Minimal Separating Closed or Compact Function Sets}\label{S:minimalgeneral}
\par\bigskip\noindent
It is always natural to expect more interesting assertions about more rigid structures. With this in mind, we will next apply the minimality requirement within narrower classes of function sets. For convenience, let us recall the definitions from the introduction section.

\par\bigskip\noindent
\begin{defin}
We say that $F\subset C_p(X)$ is a minimal compact  separating set if it is compact and separating and contains no proper compact separating subset. Similarly,  $F\subset C_p(X)$ is a minimal closed  separating set if it is closed and separating and contains no proper closed separating subset.
\end{defin}

\par\bigskip\noindent
Let us start with a non-finite example of a minimal compact separating function set.

\par\bigskip\noindent
\begin{ex}\label{ex:mincomsep}
 $X= \{0\}\cup \{1/(n+1):n\in\omega\}$ has a minimal compact separating set.
\end{ex}
\begin{proof}
Let 
$f_n(x)=1/(n+1)$ if $x=1/(n+1)$ and $0$ otherwise. Then 
$\{f_n:n\in\omega\}$ converges to $0$.  Therefore, $K=\{0\}\cup\{f_n:n\in\omega\}$ 
is compact.
The function $f_n$ is the only one that separates $1/(n+1)$ from 0, and 
every
compact subspace of $K$ that contains all $f_n$ is all of $K$.
\end{proof}

\par\bigskip\noindent
The argument of Example \ref{ex:mincomsep} can be used to prove the following statement.
\par\bigskip\noindent
\begin{pro}\label{pro:closureofminsep}
Let  $F$ be a minimal separating function set of $X$. Then $cl_{C_p(X)}(F)$ is minimal closed separating function set.
\end{pro}

\par\bigskip\noindent
Statements  \ref{ex:mincomsep} and \ref{pro:closureofminsep}  prompt the following questions.
\begin{que}
Is there a minimal closed (compact) separating function set without an isolated point?
\end{que}

\par\bigskip\noindent
\begin{que}
Does every second-countable (compact)  space have a minimal compact separating functionset?  Any Eberlein compactum?
\end{que}

\par\bigskip\noindent
\begin{que}
Does having a minimal closed (compact) separating function  set imply having a minimal separating function set?
\end{que}

\par\bigskip\noindent
\begin{que}
Is the property of having a minimal compact (closed) function set productive?
\end{que}
\par\bigskip\noindent
\begin{que}
Let $X$ have a second countable separating function set. Does $X$ have a minimal separating function set?
\end{que}

\par\bigskip\noindent
We would like to finish with two questions that were the main targets of this study and partly addressed by Proposition 
\ref{pro:discreteinsquare} and Theorem \ref{thm:criteriongeneral}.
\par\bigskip\noindent
\begin{que}\label{que:main}
Is it true that $C_p(X)$ has a minimal separating $\tau$-sized subset if and only if $X^2\setminus \Delta_X$ has a closed discrete $\tau$-sized subset. What if  $X$ is compact? What if $X$ is zero-dimensional?
\end{que}

\par\bigskip\noindent
\begin{que}
Is it true that having a minimal function separating set is equivalent to having a minimal separating family of functionally  open sets? 
\end{que}

\par

\end{document}